\documentclass[12pt,a4paper]{article}
\usepackage{amssymb}

\usepackage{amsmath}
\usepackage{graphicx,xcolor}
\usepackage{amsmath,amssymb}%

\usepackage{amsfonts}
\usepackage{amsthm,amscd}
\usepackage[english]{babel}
\usepackage{graphicx}
\usepackage{amsmath}
\usepackage{amssymb}
\usepackage{amstext}

\newtheorem{theorem}{Theorem}

\newtheorem{definition}[theorem]{Definition}
\newtheorem{example}[theorem]{Example}

\newtheorem{proposition}[theorem]{Proposition}
\newtheorem{remark}[theorem]{Remark}

\begin{document}

\title{The generalized IFS Bayesian method and an associated  variational principle covering  the classical and dynamical cases}

\author{ Artur O. Lopes and  Jairo. K. Mengue}

\date{\today}

\maketitle

\begin{abstract}
{We introduce a general IFS Bayesian method for getting posterior probabilities from prior probabilities,  and also a generalized Bayes' rule, which will contemplate a dynamical, as well as a non-dynamical setting. Given a loss function ${l}$, we detail the prior and posterior items,   their consequences and exhibit several examples. 
Taking $\Theta$ as the set of parameters and $Y$ as the set of data (which usually provides {random samples}),
a general IFS is a measurable map $\tau:\Theta\times Y \to Y$, which can be interpreted as a family of maps $\tau_\theta:Y\to Y,\,\theta\in\Theta$.
The main inspiration for the results we will get here comes from a paper by Zellner (with no dynamics), where Bayes' rule is related to a principle of minimization of {information.} 
We will show that our IFS Bayesian method which produces posterior probabilities (which are associated to holonomic probabilities) is related to the optimal solution of a variational principle, somehow corresponding to the pressure in Thermodynamic Formalism, and also to the principle of minimization of information in Information Theory.  Among other results, we present the prior dynamical elements  and we derive the  corresponding posterior elements via the Ruelle operator of Thermodynamic Formalism; getting in this way a form of dynamical Bayes' rule.
}

\end{abstract}

\section{Introduction} \label{int}

 As far as we know, no previous work  described  the existence of a connection between Bayes' rule in probability theory and the Ruelle Operator in ergodic theory. This is one of  the main goals in the present work. It is also important to remark that the generalized IFS Bayesian method as introduced in section \ref{measIFS} (see definition \ref{maindefinition}) is so broad that it encompasses the classical Bayes' rule  as well as the dynamical case (obtained via the use of the Ruelle operator).  This assertion is evidenced by several examples (see examples \ref{maine}, \ref{maine2}, \ref{trite} and \ref{exholonomic}). Furthermore, in section \ref{epi} we present a variational principle and exhibit its connections with the generalized IFS Bayesian method as introduced in section \ref{measIFS} (see Theorem \ref{teo_eq_pressure}). We exhibit examples which prove that Theorem \ref{teo_eq_pressure}  provides also a connection between the main result present in \cite{Ze} (see example \ref{zeze}) and a well known result in ergodic theory concerning the variational principle of pressure and the Ruelle Operator (see example \ref{pressureholonomic}). Holonomic probabilities for IFS played an important role in our reasoning.

We initially suppose that $\Theta$ and $Y$ are measurable spaces. In the Bayesian setting one is interested in  probabilities (prior and posterior) and minimization problems which are in some way related to a loss function ${l}:\Theta\times Y\to\mathbb{R}$.
  In its simplest formulation, { by considering finite sets,} Bayes' rule is obtained as follows.
 Suppose we have a joint probability distribution of the variables $\theta\in \Theta,y\in Y$ and the individual  probabilities of $\theta$ and $y$.
 We are interested in a convenient  expression for the probability of $\theta$ conditioned on $y$.
 In this direction, we recall that the conditional probability of $\theta$  given $y$ is
 \begin{equation} \label{BR} P(\theta|y) = \frac{P(\theta ,y)}{p(y)}.
 \end{equation}
 On the other hand,  the conditional probability of $y$ given $\theta$ is
 \begin{equation} \label{BR0}P(y|\theta) = \frac{P(\theta,y)}{p(\theta)}.
 \end{equation}
 From equations \eqref{BR} and \eqref{BR0}, the conditional  probability  of $\theta$, given $y$, can be written as:
 \begin{equation} \label{BR1} P(\theta|y) = \frac{P(y|\theta) \, p(\theta)}{p(y) },
 \end{equation}
 which is known as Bayes' rule.

 This expression is quite useful in a large number of applications.
 In the Bayesian point of view,  $p(\theta)$ is the prior probability of $\theta$, and for any given new information (a sample $y$) it is updated by a new probability $P(\theta|y)$, which is called the posterior probability of $\theta$, given $y$.

 	One way to describe the difference between the Bayesian point of view and that of classical statistics (also called frequentist) is that in the former the probability is random and the sample is fixed, and in the latter, the probability is fixed and the sample is random.

  The frequentist point of view of probability relies on the average of possible samples.  Against this point of view, it is necessary to say that, in applied problems, sometimes you don't have access to a large number of samples
 	in order to estimate the mean value of an observable.

 From the Bayesian point of view it is  more appropriate to think
 	of a probability as a measure of a prior belief. Bayesian inference relies on a given
 	observed sample and not on the result of averaging over the sample space.

 { The reader may eventually get confused because we will talk about ergodic and dynamic issues - which somehow have a frequentist character - within a setting that in principle would be Bayesian. But this is quite natural, as we shall see, in our general framework.
 	
 	Quoting \cite{Abra}, page 128:
 	
 	\begin{center}\textit{
 ``Although the two approaches are conceptually different, they are nevertheless not  ``complete strangers'' to each other and may benefit from cooperation. Frequentist analysis of Bayesian procedures can be very useful for better understanding and validating their results. Similarly, Bayesian interpretation of a frequentist procedure is often helpful in providing the intuition behind it.''}\end{center}}

 In  Section \ref{measIFS} we will present a general IFS Bayesian method associated with a given loss function ${l}$. With this purpose we consider a  more general setting using iterated function systems (IFS) and holonomic probabilities (see  Definition \ref{klr} or  \cite{MO}, \cite{LO}),  which in some sense will  play the role of the posterior probability.  
 
 A general IFS is a measurable map $\tau:\Theta\times Y \to Y$. Usually we interpret $\tau$ as a family of maps $\tau_\theta:Y\to Y,\,\theta\in\Theta$, indexed by $\theta\in \Theta$.
 The classical Bayesian point of view can also be considered as a particular case of this setting.
 Alternatively, as a particular - and a dynamical - example of IFS one can consider the inverse branches of   the shift map  $\sigma$  acting on the symbolic space $Y=\{1,2,..,d\}^\mathbb{N}$. With this purpose, it is natural to consider $\Theta=\{1,2,...,d\}$, and
 the IFS $\tau$, where $\tau_\theta: Y \to  Y$ satisfies $\tau_\theta (y_1,y_2,..,y_n,...) = (\theta, y_1,y_2,...,y_n,...)$.

{ From another point of view,} in \cite{Ze}, relations were exhibited between an optimal information processing rule and the posterior Bayes probability.
In section \ref{epi} we extend to IFS and holonomic probabilities the main result in \cite{Ze}.  The results in \cite{Ze} are not of a dynamical nature but can be considered as a particular case of our reasoning as explained in Example \ref{zeze}. { This extension allows us to emphasize  the relationship } between the optimal information processing rule and the posterior  probability, described in \cite{Ze}, with the variational principle of pressure and equilibrium measures in Thermodynamic Formalism  (see Example \ref{trite}).

 Our main results concern the description of the IFS Bayesian method, in Section \ref{measIFS},  and its connection with a variational principle, which will be presented   in Section \ref{epi}. Such a connection could be seen, in some sense,  as similar to the relation of the results for the Ruelle operator (about
	eigenfunctions and eigenprobabilities) and the variational principle of pressure, which is a major result in Thermodynamic Formalism.  For related results see \cite{Nobel1} and \cite{LLV}.

Initially, in section \ref{meas} we will consider Bayes' rule in the measurable setting with more comments and examples concerning the Bayesian method. In section  \ref{measIFS} we present the IFS  Bayesian method with examples, which explain why  our version of Bayes' rule includes the usual one, as presented in section \ref{meas}. At the end of section  \ref{measIFS} we discuss an example - illustrating our reasoning - which is  related to Markov chains. In section \ref{dynIFS} we  present examples of a dynamical nature  regarding compact spaces and contractive IFS.  In section \ref{epi} we present a variational principle and its relation with the generalized IFS Bayesian method, given in Theorem \ref{teo_eq_pressure}.

\section{Bayes' rule in the  measurable setting} \label{meas}

In this section, we propose  to present briefly  the rule of Bayes in the measurable setting using a notation that is compatible  (as much as possible) with the one in \cite{Ze}.

Let us fix a probability $d\theta$ on $\Theta$ and a measure (which can also be a probability) $dy$ on $Y$, where $\Theta$ and $Y$ are measurable spaces. We fix  a prior probability density function $\pi_a:\Theta \to (0,+\infty)$, that is,  we assume $\int \pi_a(\theta)\,d\theta = 1$, and also a  measurable, bounded and positive function $f(y,\theta)=f(y|\theta)$.

{ From the Bayesian point of view it is natural to assume, for each parameter $\theta\in \Theta$,  that $y\to f_\theta(y)=f(y|\theta)$ is} a probability density function on $Y$, that is, it satisfies
\begin{equation} \label{ljk}\int f(y,\theta)dy  = 1,\,\,\forall \theta \in \Theta.
\end{equation}
 It is also natural to assume that the map $\theta \to f_\theta$ is an injective map.

 As an example,  consider an observer which  has uncertainty (or lack of knowledge) about the probability that
 { models} a certain random problem, but, for some reason, believes that the family $f_\theta (y) dy$ is a relevant class to consider. Moreover, the observer also has faith that some parameters $\theta$ are more suitable than others for the modeling, and then, the prior probability $\pi_a (\theta) d \theta$ in $\Theta$ will describe such a conviction.

 The Bayesian point of view contemplates the modeling of a random source that produces random samples $y$, but there exists uncertainty about which probability produces the data sample. This can be expressed in the following way:
each $f_\theta(y) dy$ determines a randomness in $y$, but as there exists also uncertainty in $\theta$,
which {in the initial belief on the observer}  is described by $\pi_a(\theta) d \theta$, it is necessary to consider a mean of this randomness for $y$ which is described  by the probability density function
\begin{equation} \label{ljk3} p(y) = \int   f_\theta(y)\,\pi_a (\theta) d\theta.
\end{equation}
In some cases, the prior probability is information that the observer has at his disposal and is not the result of subjectivity.

For each fixed $y\in Y$, the function $\theta \to f(y,\theta)$,  is usually  called the likelihood function. In order to align with the notation of \cite{Ze}, we can use the notation ${l}(\theta|y)=f(y|\theta)$ to stress the dependence on $\theta$ (given $y$).

Given the above prior elements, we can exhibit  in a natural way an associated  probability on $\Theta\times Y$. Let $\pi$ be the probability on $\Theta\times Y$ defined by
\begin{equation} \label{ljk4}d\pi := f(y|\theta)dy\,\pi_a(\theta)d\theta=f_\theta (y) dy\,\pi_a(\theta)d\theta ,
\end{equation}
that is, for any measurable and finite function $g:\Theta\times Y\to\mathbb{R}$,
\[\int g(\theta,y)d\pi(\theta,y) =   \iint g(\theta,y)f(y|\theta)dy\,\pi_a(\theta)d\theta.\]
Denote $d\nu := \pi_a(\theta)d\theta.$ It is easy to show that the $\theta-$marginal of $\pi$ is $\nu$. Furthermore,
 looking for the definition of $\pi$ we get that $f(y|\theta)dy$ is a probability kernel and  the right-hand side of equation \[d\pi = [f(y|\theta)dy]d\nu(\theta)\] is a disintegration of $\pi$.

{The random choice of $y$ according to the family of probabilities $f_\theta(y) d y$ is known in Statistics as the sampling procedure.  Each possible $y$ is   a latent sample.}
A  classical strategy in Bayesian inference is  the procedure of updating  a prior
belief distribution $\pi_a$  on $\Theta$ to a posterior distribution $\pi_p$ on $\Theta$, when the parameter of interest is connected
to observations (taken from  random samples $y\in Y$), via the likelihood function. All this is in some way related to Bayes' rule.

\bigskip

In order to get Bayes' rule it is natural  to consider a disintegration of $\pi$ in the opposite order of variables.
By integrating functions of variable $y$ we get that the $y-$marginal of $\pi$ is given by $d\mu:=p(y)dy$, where
\begin{equation}\label{normalizerp}
	p(y) = \int f(y|\theta)\pi_a(\theta)d\theta.
	\end{equation}
Finally, in order to complete the disintegration of $\pi$, we need to find a probability kernel $k$ such that
$d\pi = k(\theta|y)p(y)dy.$
With this purpose, just return to the definition of $\pi$ in order to get
\begin{equation}\label{posteriorBayes0}k(\theta|y) = \frac{f(y|\theta)\pi_a(\theta)}{p(y)}d\theta.
\end{equation}

\medskip

It is clear that for each given $y_0\in Y$,  the probability $k(\theta|y_0)$ is absolutely continuous with respect to $d\theta$. The associated probability density function is given by
\begin{equation}\label{posteriorBayes}
	\pi_p(\theta|y_0) := \frac{f(y_0|\theta)\pi_a(\theta)}{p(y_0)},
	\end{equation}
which will be  called the posterior probability density function on $\theta$ associated with the observed sample $y_0$. In this way $\pi_a(\theta)d\theta$ is the prior probability on $\Theta$ while $\pi_p(\theta|y_0)d\theta$ is the posterior probability on $\Theta$, given the  observed sample $y_0$.

The formula
\begin{equation}\label{posteriorBayes3}\pi_p(\theta|y) = \frac{f(y|\theta)\pi_a(\theta)}{p(y)}
\end{equation} (which is constructed from a change of order of variables in the disintegration) corresponds to Bayes' rule in the measurable setting.

In the Bayesian method, $\pi_a(\theta)$ describes the initial (prior) randomness of the probabilities indexed by $\theta$,  then, given the observed sample $y_0$, we update the randomness of the probabilities indexed by $\theta\in \Theta$, via the posterior $\pi_p(\theta|y_0).$

 Finally, returning to  the computations, we observe that from the disintegration of $\pi$,
\begin{equation}\label{piBayes} [f(y|\theta)dy]\pi_a(\theta)d\theta =d\pi = [\pi_p(\theta|y)d\theta] p(y)dy\end{equation}
we get a map from the  prior to the  posterior via Bayes' rule, which can be  illustrated by the following diagram
(see also figure 1 in \cite{Ze}):
\[ \begin{array}{c} f(y|\theta) \\ \pi_a(\theta) \end{array} \longrightarrow \begin{array}{c} \pi_p(\theta|y) \\ p(y) \end{array}.\]

\begin{remark}\label{rem}
From a theoretical point of view, even in the case where $f(y|\theta)=f_\theta(y)$  is not a family of probability density functions on $Y$, we get that, for each fixed point $y_0\in Y$, the function $\theta\to \pi_p(\theta|y_0)$, defined from \eqref{posteriorBayes} is a probability density function on $\Theta$. It could be called the posterior probability density function defined by the prior probability density function $\pi_a$ and the point $y_0\in Y$.
 An important issue when analyzing equation \eqref{posteriorBayes} is that formally for each fixed point $y_0\in Y$, usually, the function $\theta \to f(y_0|\theta)$   is not a probability density function on $\Theta$. So, we can interpret that \eqref{posteriorBayes} exhibits a normalization of $f$ in order to get a new function $\pi_p$ which is a probability density function on $\Theta$.

 Finally, we observe that, in order to assume that $f(y|\theta)$ is a probability density function on $Y$ for each $\theta$, it is also necessary to fix ({or to choose}) a measure $dy$ on $Y$ (see \eqref{ljk}). Such a $dy$ is not used in \eqref{posteriorBayes}.

	\end{remark}

The idea behind this procedure is as follows: we are not sure, a priori (in $\theta$), about which density $f_\theta$ is responsible for the generation of the sample $y$. This uncertainty is described by the prior $\pi_a$ on $\theta$ (this choice of $\pi_a$ may be subjective, or not, and sometimes  results from the previous intuition   of the observer).  In some examples taken from Physics, the description of this uncertainty on $\theta$ is a consequence of  the nature of the physical problem.
There are two possible points of view that are popular for applications.

I. We assume that one particular probability  associated with $f_{\theta_0}$, where $\theta_0$ is fixed among the $\theta \in \Theta$, is generating the sampling procedure. However, we do not know which $\theta_0$ produces the randomness.
Then, having obtained $y_0$ by sampling, we can ask (an inference) what would be the  probability (indexed $\theta$) that more likely fits the data obtained from the sample $y_0$.  Given $y_0$, the maximum $\theta_0$ in $\theta$ of the  likelihood $\theta \to {l}(\theta|y_0)$, in some sense, indicates  the parameter $\theta_0$  which  seems more likely  to match the data (see the beginning of Section 1.2 in \cite{Abra}). As the sampling procedure of getting $y_0$ is random, this value alone $\theta_0$ does not deterministically determine the solution to the problem;  it is natural to ask about the probability in $\theta$, of the $f_\theta$ indexed by $\theta$, which best  fit the data, given the sample $y_0$.
This will be described by $\pi_p$ in \eqref{posteriorBayes}. The main goal in this setting is to try to guess $\theta_0$ from the sample $y_0.$ This type of point of view in a dynamical setting is considered in  \cite{LLV} (in this case the sample $y_0$ is  obtained via a Birkhoff point where the  probability associated to $f_{\theta_0}$ is assumed to be ergodic).
\medskip

II.  There exists uncertainty in the $f_\theta$ that produces the sample $y_0$,  and this uncertainty is described by the prior probability on $\theta$ given by $\pi_a(\theta)$. The  prior probability $\pi_a$ (which for some reason the observer  believes is natural for modeling a certain real-life application)  on  $\theta\in \Theta$, has to be  updated, from the information given by the sample data $y_0$ - which we get from the specific applied problem under consideration -  in order to get another probability $\pi_p$, which suits best the model
under analysis (which will be described by \eqref{posteriorBayes}). This kind of problem appears in Decision Theory where it is usual to consider the case where the support of $\pi_a$ is the finite set $\{\theta_1, \theta_2\}$. As an example (to be more elaborated in Example \ref{edr}) one can consider a case where there exists a probability $1/3$ that the sample $y_0$ was obtained from $f_{\theta_1}$, and a probability $2/3$ that the sample $y_0$ was obtained from $f_{\theta_2}$ (these are the conditions for a simple Hypothesis test according to the Appendix in \cite{DLL}). This point of view in a dynamical setting was considered in Section 6 in \cite{FLL} and also in \cite{DLL}.

\smallskip

\smallskip

\begin{example} \label{popo}
 Let $\Theta = [0,1]$ and for each $\theta \in [0,1]$ we define the probability $p_\theta = (\theta,1-\theta)$ on $\{0,1\}$ (that is $p_\theta(0)=\theta$ and $p_\theta(1) = 1-\theta$) and the corresponding independent Bernoulli probability (using the same notation) $p_\theta$ on $\Omega =\{0,1\}^{\mathbb{N}}$. Suppose that in a game, balls with numbers zero or one are randomly picked following an unknown probability $p_{\theta_0}$. Let $Y=\{0,1\}^{1000}=\{(y_1,y_2,...,y_{1000})|y_i\in\{0,1\}\}$, that is, the set of sequences of zeros and ones of length 1000.

 For each $\theta \in \Theta$ consider the function $f(y|\theta) = \theta^{\#0}(1-\theta)^{\#1}$ where $\#0$ and $\#1$ denote the numbers of zeros and ones in the finite sequence $y=(y_1,y_2,y_3,...,y_{1000})$, respectively.  We observe that $f(y|\theta)$ describes for each fixed $\theta$ the probability of getting the sequence $y$ after playing 1000 times the game, following the probability $p_\theta$ and supposing independence and identical distributions. As we mentioned above, the probability $p_{\theta_0}$ is unknown. Consider the uniform Lebesgue probability $d\theta$ on $\Theta=[0,1]$. Let $\theta\to \pi_a(\theta)$ be a probability density function and suppose that we have a prior belief that $\theta$ belongs to any interval $I$ with probability $\int_I\pi_a(\theta)d\theta$. This is before the game starts.

 Now suppose that we just play the game 1000 times and observe the results getting a sample $\tilde{y}=(y_1,y_2,...,y_{1000})$. Suppose also that for this {observed sample} we have $\#0 = 900$ and $\#1 = 100$. Such {an observed} sample can be used to update our prior belief concerning $\theta$. For this purpose, we can use Bayes' rule. The posterior belief is that now $\theta$ belongs to any interval $I$ with probability given by $\int_I\pi_p(\theta|\tilde{y})d\theta$ where
 \[     \pi_p(\theta|\tilde{y}) = \frac{f(\tilde{y}|\theta)\pi_a(\theta)}{\int f(\tilde{y}|\theta)\pi_a(\theta)\,d\theta} =\frac{\theta^{900}(1-\theta)^{100}\pi_a(\theta)}{\int \theta^{900}(1-\theta)^{100}\pi_a(\theta)\,d\theta} .\]

\end{example}

\begin{example} \label{edr}  A random source is modeled in the following way:   samples are obtained in $Y=\{1,2\}$, where the prior $\pi_a$ on $\theta\in \Theta =\{\theta_1,\theta_2\}$, is given by
$\pi_a(\theta_1) = 1/3$ and $\pi_a(\theta_2)=2/3$.

$y\to f(y|\theta)=f_\theta(y)$ are assumed to be  independent probabilities on
$Y=\{1,2\}$, for each $\theta$,  and they are given in the following way:

a)  $f_{\theta_1}(1) =f(1|\theta_1 )= 3/10$ and $f_{\theta_1}(2)= f(2|\theta_1)= 7/10,$

b) $f_{\theta_2}(1) =f(1|\theta_2 )= 4/10$ and $f_{\theta_2}(2) =f(2|\theta_2)= 6/10.$

Then, {by considering the prior $\pi_a$ to $\theta$,}   the probabilities of $1$ and $2$ are given by
$$ p(1)= \pi_a(\theta_1)  \,\,\, 3/10 + \pi_a(\theta_2)   \,\,\, 4/10= 1/3 \,\,\, 3/10 + 2/3  \,\,\, 4/10 = 11/30,$$
and
$$ p(2)= \pi_a(\theta_1)  \, \,\, 7/10    + \pi_a(\theta_2)  \,\,\,  6/10= 1/3 \, \,\, 7/10    + 2/3  \,\,\,  6/10 = 19/30.$$

The probability $p(y)$ satisfies
$$ p(y) = \sum_\theta f(y|\theta) \pi_a (\theta).$$

The value $p(1)= 11/30$ describes the probability of getting $1$ as a sample obtained from the source by considering the prior $\pi_a$ to $\theta$.

Now we can ask about the posterior probability of $\theta$, given $y$, which will be denoted by $\pi_p(\theta|y)$. Suppose, for example,  we get $y=1$ as a sample. In this case, what is the posterior  probability  $\pi_p(\theta_2|1) $ that
this sample was obtained from $\theta_2$?

Bayes' rule determines
$$ \pi_p(\theta|y) = \frac{f(y|\theta) \, \pi_a (\theta) }{p(y) } .$$

Then,
$$\pi_p(\theta_1|1) = \frac{3/10\,\, 1/3}{11/30}=3/11\,\,\text{and}\,\,\pi_p(\theta_2|1)  = 8/11,$$
while
$$\pi_p(\theta_1|2) = \frac{7/10\,\, 1/3}{19/30}=7/19\,\,\text{and}\,\,\pi_p(\theta_2|1) =  12/19.$$

{ Once a sample $y_0$ has been obtained, we can update the uncertainty on $\Theta$, given by $\pi_a(\theta)$, by taking into account the new probability $\pi_p (\theta|y_0)$, as  in \eqref{posteriorBayes}.}
\end{example}

\section{Bayes' rule  and the IFS Bayesian method} \label{measIFS}

In this section, we consider another point of view,  which is associated with a measurable IFS $\tau$.  We will present a general {\bf IFS Bayesian method} which will work in a dynamical setting as well as in a non-dynamical setting. In this section, we consider that $\Theta$ and $Y$ are just measurable spaces and the IFS $\tau$ is measurable. This section is inspired by \cite{MO}, which considers a more particular case where $\Theta,Y$ are compact metric spaces and the IFS is contractible (which will be presented as  an example in Section \ref{dynIFS}).

The next example  contains fundamental ideas for a better  understanding of our reasoning in this section.

\begin{example}\label{meansample} This example proposes a generalization of formula \eqref{posteriorBayes}, which is Bayes' rule. The idea is to  replace the  observed sample $y_0\in Y$, by  a
probability $\rho$ on $Y$, in order to define a new kind of posterior probability density function which will be denoted by $\pi_p(\theta|\rho)$.
	As motivation, consider two points $y_1,y_2\in Y$ which play the role of two of the possible {samples}. After one observation, if we get the sample $y_1$ then we can update our prior belief of the probability on $\theta$ from the prior $\pi_a(\theta)d\theta$ to the posterior $\pi_p(\theta|y_1)d\theta$, where $\pi_p(\theta|y)$ is given by Bayes' rule \eqref{posteriorBayes3} as presented in section \ref{meas}. Similarly, if we observe the sample $y_2$ then we update it as another posterior, given by  $\pi_p(\theta|y_2)d\theta$.  Suppose now that 10 observations are made and as result the sample $y_1$ was obtained 3 times while the sample $y_2$ was obtained 7 times. In this case unfortunately we possibly have an uncertainty concerning the {sample}. Our uncertainty can be described by a probability $\rho = \frac{3}{10}\delta_{y_1}+\frac{7}{10}\delta_{y_2}$: we get a sample $y_1$ with probability $3/10$ and a sample $y_2$ with probability $7/10$ { (and this determines  the $\rho$ in the specific problem we are considering)}.

\smallskip

An observer of certain random phenomena realizes that the $\rho$ obtained above  is in agreement with the mathematical modeling he is looking for.  Assume this person is not  interested
(or, does not consider relevant for the problem) in getting posterior information on $\theta$ based on a single random sample, but the relevant information is in fact on the mean value of the posterior probability based on the global and random information given by $\rho$.

	 Therefore, alternatively, it is possible to update his belief by {defining a mean posterior}
	\[\pi_p(\theta|\rho):= \frac{3}{10}\pi_p(\theta|y_1)+\frac{7}{10}\pi_p(\theta|y_2)=\int \pi_p(\theta|y)\,d\rho(y),\]
	where $\pi_p(\theta|y)$ is given in \eqref{posteriorBayes3}.
	
	More generally, given any probability $\rho$ on $Y$ (which can be considered as a new information concerning the distribution of $y$ obtained by observations) we can update our prior $\pi_a$ to get the posterior  $\pi_p(\theta|\rho)$ defined by 	\begin{equation}\label{meansample2}
		\pi_p(\theta|\rho) :=      \int \frac{f(y|\theta)\pi_a(\theta)}{p(y)}d\rho(y)=\int \pi_p(\theta|y)\,d\rho(y).
	\end{equation}
	Clearly, given a sample $y_0\in Y$ and considering $\rho=\delta_{y_0}$ we get from \eqref{meansample2} the  Bayes' rule $$\pi_p(\theta|{y_0}) =  \frac{f(y_0|\theta)\pi_a(\theta)}{p(y_0)}.$$
	
	If, for example, the new information concerning the distribution of $y$, which is an observed distribution $\rho$, coincides with our prior belief $p(y)dy$ (that is, if $d\rho(y) = p(y)dy$) it is natural to expect that $\pi_p=\pi_a$. It is in accordance with the computation below
	\[\pi_p(\theta|p(y)dy) = \int_Y \frac{f(y|\theta)\pi_a(\theta)}{p(y)}p(y)dy = \int_Y {f(y|\theta)\pi_a(\theta)}dy = \pi_a(\theta).\]
	
\end{example}

From now on we will not update the prior $\pi_a$ to a posterior $\pi_p$ just from a single-sample procedure. The newer $\pi_p$ will be constructed from theoretical information concerning an IFS together  with other mathematical concepts.  We will provide examples that support the claim  that the classical approach presented in section \ref{meas} can be seen embedded in this new theoretical point of view.

Given a loss function ${l}: \theta \times Y\to \mathbb{R}$ and the IFS $\tau$, our main purpose in this section is to detail the prior and posterior items of the  general IFS  Bayesian method to be defined below.
Along the section, we will present examples to illustrate our reasoning.

The particular IFS described in the next two definitions will help to justify our point of view.

\begin{definition} The  particular case where the IFS is constant (that is, there exists $y_0\in Y$, such that $\tau_\theta(y)=y_0,\,\forall \theta,y$), will be called the {\bf constant IFS} associated to $y_0$ in $Y$.

\end{definition}

\begin{definition}   When the IFS is such that it is given by $\tau_\theta(y) = y,\,\forall\theta,y$, we call it the {\bf identity IFS}.
\end{definition}

 The identity IFS corresponds in some sense to the case where there is no dynamics at all.  We will show that our formalism corresponds to the  classical Bayes' rule when considering the constant IFS and also when considering the identity IFS.

\begin{definition} Given a prior probability $\nu$ on $\Theta$ and a measurable function $g:\Theta\times Y\to[0,+\infty)$  we will call $g$  {\bf a Jacobian} with respect to $\nu$ (or $\nu$-Jacobian) if
\begin{equation} \label{wqe}\int g(\theta,y) \,d\nu(\theta) = 1,\,\forall y \in Y.
\end{equation}
\end{definition}

We can interpret a $\nu$-Jacobian  $g$ as a family, indexed by  $y$,  of  probability density functions $\theta \to g_y(\theta)=g(\theta,y)$, of the variable $\theta$. In the above section \ref{meas} the function $\pi_p(\theta|y)$ plays the role of $g$. As in the above section, usually, we start with a non-Jacobian function (which eventually is a family of densities in the opposite variable - a likelihood). Our first objective is to describe what will be considered a normalization in the present setting. Remark \ref{rem} is useful to the reader at this point.

\begin{definition} \label{denor} Given a measurable and bounded function $l:\Theta\times Y\to(0,+\infty)$, an IFS $\tau$ and a probability $\nu$ on $\Theta$, we will call a \textbf{ normalizer pair} for $(l,\nu,\tau)$, a pair of measurable and positive functions $(\varphi,\psi)$, defined on $Y$,  such that the function
	\begin{equation}\label{normalizer}
		\bar{l}(\theta,y):= \frac{{l}(\theta,y)\cdot \psi(\tau_\theta(y))}{\psi(y)\cdot\varphi(y)}
	\end{equation}
	is a $\nu$-Jacobian.
\end{definition}

\begin{remark} \label{qwe} Note that by taking  $\psi \equiv 1$ and $\varphi(y) = \int {l}(\theta,y)d\nu(\theta)$, we always can get a normalizer pair satisfying Definition \ref{denor}, which we will call the {\bf canonical normalizer pair}. We point out that in some cases a different choice of  normalizer pair is more appropriate (for instance it happens in Example \ref{marma}).
\end{remark}

In section \ref{meas} the function $p$ defined in \eqref{normalizerp} plays the role of $\varphi$ in the canonical normalizer pair.   The next proposition shows that once we choose $\psi$ the function $\varphi$ is also determined.

\begin{proposition} \label{normalize}
	Given a measurable, positive and bounded function $\psi$ we have that  $(\varphi,\psi)$ is a normalizer pair for $(l,\nu,\tau)$ iff
	\begin{equation}\label{varphi}
	\varphi(y) = \frac{1}{\psi(y)}\int {{l}(\theta,y)\cdot \psi(\tau_\theta(y))}\,d\nu(\theta).
	\end{equation}
	Furthermore, the associated Jacobian function $\bar{l}$ given in \eqref{normalizer} satisfies
	\begin{equation}\label{barl}
		\bar{l}(\theta,y) = \frac{l(\theta,y)\psi(\tau_\theta(y))}{\int l(\theta,y)\psi(\tau_\theta(y))\,d\nu(\theta)}.
	\end{equation}
\end{proposition}
\begin{proof}
Let $(\varphi,\psi)$ be a normalizer pair, and suppose $\psi$ is bounded.	We observe that \[\int \frac{{l}(\theta,y)\cdot \psi(\tau_\theta(y))}{\psi(y)\cdot\varphi(y)}\,d\nu(\theta) = 1, \,\,\,\forall\,y\]
	corresponds to
	\[\int {{l}(\theta,y)\cdot \psi(\tau_\theta(y))}\,d\nu(\theta) =\psi(y)\cdot\varphi(y),\,\,\,\forall\,y.\]
	Therefore we get \eqref{varphi}. On the other hand, given a measurable, positive, and bounded function $\psi$, by defining $\varphi$ from \eqref{varphi} we get a normalizer pair. Furthermore,
	\[\bar{l}(\theta,y):=\frac{{l}(\theta,y)\cdot \psi(\tau_\theta(y))}{\psi(y)\cdot\varphi(y)}=\frac{{l}(\theta,y)\cdot \psi(\tau_\theta(y))}{\int {{l}(\theta,y)\cdot \psi(\tau_\theta(y))}\,d\nu(\theta)}.\]
	
\end{proof}
 	
The meaning of the normalizer pair  $(\varphi,\psi)$ in the above definition is not the one considered in \cite{MO}, where $\varphi$ is assumed to be constant. As we will show in Section \ref{dynIFS}, a special normalizer pair, given by a constant function $\varphi$, occur when the IFS is contractible (see Definition \ref{contr}) and the spaces are compact metric spaces (in this case $\varphi$ is the main eigenvalue for a Transfer Operator). In \cite{MO} such as setting was considered.

In the case of  the {\it constant IFS} associated to $y_0$ in $Y$, we can  consider a special normalization  where $\varphi$  is taken constant, as we will  show in  the next example.
\begin{example} \label{kaj}
	Suppose $\tau_\theta(y) = y_0,\,\forall\theta,y$. Then we get 	$\psi(\tau_\theta(y))/\psi(y)=\psi(y_0)/\psi(y)$ for any $\psi$. Considering \eqref{normalizer} and choosing $\psi(y) = \int {l}(\theta,y)d\nu(\theta)$ and $\varphi(y)\equiv \psi(y_0)$ we get
	$$\int \frac{{l}(\theta,y)\psi(\tau_\theta(y))}{\psi(y)\varphi(y)}d\nu(\theta)=\frac{\int {l}(\theta,y)\,d\nu(\theta)}{ \psi(y)}=\frac{\int{l}(\theta,y)d\nu(\theta)}{\int {l}(\theta,y)d\nu(\theta)}=1 \,\forall\,y.$$ So $(\varphi,\psi)$ is a normalizer pair with $\varphi$ constant.
\end{example}

\begin{remark} \label{retp} Considering the normalizer pair of the above example we get that $$\frac{\psi(\tau_\theta(y))}{ \psi(y) \varphi(y)} = \frac{1}{\psi(y)} =\frac{1}{\int {l}(\theta,y)d\nu(\theta)} $$ corresponds to the normalizer function $\frac{1}{p}$ in  \eqref{normalizerp} and also to the normalization  $c$ on page 279 in \cite{Ze}. Furthermore, as $\frac{\psi(\tau_\theta(y))}{\psi(y)\varphi(y)}=\frac{\psi(y_0)}{\psi(y) \varphi(y)}$ is just a new function of the $y-$variable, the normalizer pair is essentially unique, that is, for any normalizer pair $(\varphi,\psi)$ we get $\frac{\psi(\tau_\theta(y))}{\psi(y)\varphi(y)} =\frac{1}{\int {l}(\theta,y)d\nu(\theta)}$. Another possible normalizer pair is given, for example, by $\psi=1$ and $\varphi(y) =\int {l}(\theta,y)d\nu(\theta))$ (the canonical normalizer pair), but we prefer to highlight the choice of $\varphi$ constant because it is similar to the normalization in \cite{MO}. 
\end{remark}

For a general IFS, it is not always possible to get normalizer pairs where $\varphi$ is constant.
This is the case for the identity IFS, as we will show next. 

\begin{example} \label{idi1} Suppose $\tau_\theta(y) = y,\,\forall\theta,y$. Then, as $\psi(\tau_\theta(y))=\psi(y)$, by considering \eqref{normalizer} we get
	$$1=\int \frac{{l}(\theta,y)\psi(\tau_\theta(y))}{\psi(y)\varphi(y)}d\nu(\theta)=\frac{\int {l}(\theta,y)\,d\nu(\theta)}{ \varphi(y)}=\,\forall\,y.$$
There exists a unique possible $\varphi$ which is $\varphi(y) = \int {l}(\theta,y)d\nu(\theta)$. In this case, we can not normalize ${{l}}$ with $\varphi$ constant.
\end{example}

\begin{remark}\label{p} As we saw, considering the above two examples ($\tau_\theta(y)=y_0$ or $\tau_\theta(y)=y$), there exists essentially one normalizer pair, that is, $$\frac{\psi(\tau_\theta(y))}{\psi(y)\varphi(y)} =\frac{1}{ \int {l}(\theta,y)d\nu(\theta)}.$$
It is easy to show that in general, if $\tau$ does not depend on $\theta$, this property is true. In this case, considering equation \eqref{barl}, the Jacobian function is given by
\[\bar{{l}}(\theta,y)= \frac{{l}(\theta,y)}{ \int l(\theta,y)d\nu(\theta)}.\]
 The expression above  reminds us of equation \eqref{posteriorBayes3}.

\end{remark}

An IFS can also be  seen as a dynamical system where  concepts like entropy and others can be defined.
Under such a point of view, when considering the IFS setting, it is natural to replace the concept of invariant probability (the classical dynamical point of view) with the concept of holonomic probability (see \cite{LO} and \cite{MO}).

\begin{definition}\label{klr} We say that a probability $\pi$ on $\Theta\times Y$ is \textbf{holonomic} with respect to the IFS $\tau$ if for any measurable and bounded function $g:Y\to\mathbb{R}$,
\begin{equation} \label{krasp0}\int g(y)\,d\pi(\theta,y) = \int g(\tau_\theta(y))d\pi(\theta,y).
\end{equation}
\end{definition}

\begin{definition} \label{sst} Given a $\nu$-Jacobian  $\bar{{l}}$, we say that a probability $\rho$ on $Y$ is \textbf{stationary} with respect to $(\bar{{l}},\nu,\tau)$, if for any measurable and bounded function $g:Y\to\mathbb{R}$, we have
\begin{equation} \label{krasp1}\int\int \bar{{l}}(\theta,y)g(\tau_\theta(y))\,d\nu(\theta) d\rho(y) = \int g(y)\,d\rho(y).
\end{equation}
\end{definition}

\begin{proposition}
	If $\rho$ is stationary  with respect to $(\bar{l},\nu,\tau)$, then the probability $\pi$ on $\Theta\times Y$, defined by
	$d\pi = 	\bar{{l}}(\theta,y)d\nu d\rho$, is holonomic and its $y-$marginal is the probability $\rho$.
\end{proposition}
\begin{proof}
	For any  measurable and bounded
  function $g:Y\to\mathbb{R}$ we have
	\[\int g(y)d\pi(x,y) = \int\int	\bar{{l}}(\theta,y)g(y)d\nu(\theta) d\rho(y) = \int 1\cdot g(y)d\rho(y)\] and so $\rho$ is the $y-$marginal of $\pi$.
	Furthermore, $\pi$ is holonomic because
	\[\int g(\tau_\theta(y))d\pi(x,y)= \int\int	\bar{{l}}(\theta,y)g(\tau_\theta(y))d\nu(\theta) d\rho(y) \]\[\stackrel{stationary}{=} \int g(y)\,d\rho(y) = \int g(y)\,d\pi(x,y).\]
	
\end{proof}

\begin{example} \label{kaj1}
	In the case $\tau_\theta(y)=y_0$ the unique stationary probability is $\rho=\delta_{y_0}$. Indeed, for any measurable and bounded  function $g:Y\to\mathbb{R}$ and any  stationary probability $\rho$ for $(\bar{{l}},\nu,\tau)$ we have
	\[\int g(y)d\rho(y) \stackrel{stationary}{=} \int\int \bar{{l}}(\theta,y)g(\tau_\theta(y))d\nu(\theta) d\rho(y) \]
	\[= \int\int \bar{{l}}(\theta,y)g(y_0)d\nu(\theta) d\rho(y) = \int 1\cdot g(y_0)\,d\rho(y)=g(y_0).\]
	This proves that $\rho=\delta_{y_0}$.
	Furthermore, a probability $\pi$ on $\Theta\times Y$ is holonomic iff its $y-$marginal is $\delta_{y_0}$. Indeed, by definition, $\pi$ is holonomic iff $\int g(\tau_\theta(y))d\pi = \int g(y)d\pi$ for any measurable and bounded
 function $g:Y\to\mathbb{R}$. The left-hand side is equal to $g(y_0)$ and so $\pi$ is holonomic iff $g(y_0) = \int g(y)d\pi$. Finally, for any given  $\nu$-Jacobian  $\bar{l}$ we get the holonomic probability $d\pi = \bar{{l}}(\theta,y_0)d\nu(\theta) d\delta_{y_0}(y)$.
\end{example}
\smallskip

Now we consider the identity IFS.

\begin{example} \label{idi}
	 Suppose $\tau_\theta(y)=y$,  for any $\theta$. We claim that any probability $\rho$ on $Y$ is stationary. This is not a surprise due to the fact that  $\tau_\theta(y)=y$, captures the case where there are no dynamics at all. In order to prove the claim, consider any probability $\rho$ on $Y$. We will prove that it is  stationary  with respect to $(\bar{{l}},\nu,\tau)$. For any measurable and bounded function $g:Y\to\mathbb{R}$ we have
	\[\int\int \bar{{l}}(\theta,y)g(\tau_\theta(y))d\nu(\theta) d\rho(y) = \int\int \bar{{l}}(\theta,y)g(y)d\nu(\theta) d\rho(y)= \int 1\cdot g(y)\,d\rho(y) .\]
	This proves that $\rho$ is stationary. As a consequence, for any probability $\rho$ on $Y$ we get that $d\pi = \bar{l}(\theta,y)d\nu d\rho$ is holonomic. More generally, any probability $\pi$ on $\Theta\times Y$ is holonomic. Indeed, as $\tau_\theta(y)=y$ we get
	\[\int g(\tau_\theta(y))d\pi = \int g(y)\,d\pi\]
	for any probability $\pi$ and any integrable function $g$.
\end{example}

Now, we will introduce a general IFS  Bayesian posterior method. In this way, we will consider a family of prior items which will allow us to obtain a family of posterior items.   We remember that in section \ref{meas}, for a fixed likelihood function $l(\theta|y)=f(y|\theta)$,  from a prior probability density function $\pi_a(\theta)$, and a given sample $y_0$, we get the posterior probability density function $\pi_p(\theta|y_0)$.  Now, from a prior probability density function $\pi_a$ on $\Theta$ and a probability $\rho $ over  $Y$, we will  get a new kind of posterior $\pi_p$ on $\Theta$. For a fixed  loss function ${l}$,   following a more broad point of view, we  do not have to  consider any more just the case $\pi_p(\theta|y_0)$ (on the variable $\theta$), where $y_0$ comes from probabilistic  sampling, but  instead $\pi_p(\theta|\tau,\psi,\rho)$ (on the variable $\theta$). Example \ref{meansample} illustrates this point of view. The introduction of a general $\rho $ over  $Y$, as described below,  will be an important issue in some other examples like \ref{marma},  \ref{trite} and \ref{exholonomic}.

\smallskip

\begin{definition}\label{maindefinition} The IFS Bayesian method is based on the definition of posterior items, from the information of prior and intermediate items as below:
	
\noindent
\textbf{Prior items:}\newline
1. a probability $d\theta$ on $\Theta$  \newline
2. a measurable  probability density function $\pi_a:\Theta\to (0,+\infty)$ and so the probability $d\nu(\theta) := \pi_a(\theta)d\theta$\newline
3. a measurable and bounded loss function ${l}:\Theta\times Y\to (0,+\infty)$\newline
\textbf{ Intermediate items:}     \newline
4. an IFS $\tau$\newline
5. a measurable and bounded function $\psi:Y\to (0,+\infty)$ and so a normalizer pair $(\varphi,\psi)$ for $({l},\nu,\tau)$, where $\varphi$ is given in \eqref{varphi}, and the Jacobian  (as given in \eqref{normalizer}, \eqref{barl})
\begin{equation} \label{kasp}	\bar{l}(\theta,y) = \frac{l(\theta,y)\psi(\tau_\theta(y))}{\int l(\theta,y)\psi(\tau_\theta(y))\,\pi_a(\theta)d\theta}\end{equation}
6. a stationary probability $\rho$  for $(\bar{{l}},\nu,\tau)$   in the sense of \eqref{krasp1}.

\smallskip
\noindent
\textbf{Posterior items:} \newline
1. The holonomic probability $\pi = \bar{l}(\theta,y)\pi_a(\theta)d\theta d\rho(y)$\newline
2. The probability density kernel $\pi_p(.|\tau,\psi,y):\Theta\to\mathbb{R}$ given by
\begin{equation}\label{Artur}\pi_p(\theta|\tau,\psi,y) = \bar{l}(\theta,y)\pi_a(\theta) = \frac{l(\theta,y)\psi(\tau_\theta(y))\pi_a(\theta)}{\int l(\theta,y)\psi(\tau_\theta(y))\,\pi_a(\theta)d\theta}\end{equation}
3. The probability $\nu_p$ which is the $\theta-$marginal of $\pi$\newline
4. The probability density function $\pi_p(.|\tau,\psi,\rho):\Theta\to\mathbb{R}$, which satisfies $d\nu_p(\theta) = \pi_p(\theta|\tau,\psi,\rho)d\theta$. That is,
\begin{equation}\label{IFSBayes}
	\pi_p(\theta|\tau,\psi,\rho) =\int_Y \frac{l(\theta,y)\psi(\tau_\theta(y))\pi_a(\theta)}{\int l(\theta,y)\psi(\tau_\theta(y))\pi_a(\theta)d\theta}d\rho(y).
\end{equation}

\end{definition}

 The Posterior items are uniquely determined by the Prior and Intermediate items.
The probability $\nu=\pi_a(\theta)d\theta$ is called the prior probability on the variable $\theta$. We point out that given the prior items 1, 2, 3 above, in some cases, there may exist more than one choice for the intermediate items 4,  5, and 6, and finding natural, or suitable, choices can be the main issue.  When trying to detail the items on the IFS Bayesian method, its inherent flexibility allows one to get different kinds of  information depending on these choices (which can be chosen according to  convenience or context). We will illustrate this claim via several different types of examples.

 When considering the classic Bayesian case, the information obtained from an observed sample  is  the key element for setting  the Intermediate items (see Examples \ref{maine} and \ref{maine2}).

We observe that we can always consider as intermediate the function $\psi=1$ (which corresponds to the canonical normalizer pair) and for such an intermediate item we get
\begin{equation} \label{wrt}\pi_p(\theta|\tau,1,y) =  \frac{l(\theta,y)\pi_a(\theta)}{\int l(\theta,y)\,\pi_a(\theta)d\theta}
\end{equation}
and
\begin{equation} \label{wrt1}\pi_p(\theta|\tau,1,\rho) =\int_Y \frac{l(\theta,y)\pi_a(\theta)}{\int l(\theta,y)\pi_a(\theta)d\theta}d\rho(y).
\end{equation}
Depending on the problem of interest the choice $\psi=1$ is not adequate.

\bigskip

We can say that equations \eqref{Artur} and \eqref{IFSBayes} play the role of a generalized Bayes' rule (in two different forms).
We will call $\pi_p(\theta|\tau,\psi,y) $ the posterior probability density function of the variable $\theta$ associated to the  $(\tau,\psi,y)$, and $\pi_p(\theta|\tau,\psi,\rho)$ the posterior probability density function of the variable $\theta$ associated to intermediate  items $(\tau,\psi,\rho)$.

\begin{remark} \label{mui} In order to get the usual  Bayes' rule  an optional assumption regarding the above-defined items can be taken:

- fix a measure $dy$ on $Y$ and assume ${l}$ is such that for all $\theta\in \Theta,$
\begin{equation} \label{juju}\int {l}(\theta,y) \, d y=1.
\end{equation}

{ In this case, the function $l(\theta,y)$ is a probability density function with respect to $d y$, for each fixed $\theta \in \Theta$.  Taking $\psi=1$ we get from \eqref{wrt} a Bayes' rule expression similar to \eqref{posteriorBayes3}},  where the density  $p(y)$ (with respect to $dy$) corresponds to $\int l(\theta,y) \pi_a(\theta) d \theta$ (see also \eqref{normalizerp}).

It is not a necessary hypothesis for the application of the method, as discussed in Remark \ref{rem}.
\end{remark}

 The next two examples will clarify the assertion   that the IFS Bayesian method used to construct $\pi_p$, as given in \eqref{Artur} and \eqref{IFSBayes}, extends Bayes' rule as expressed in \eqref{posteriorBayes} and \eqref{posteriorBayes3}.

\begin{example} \label{maine} Consider an observed sample $y_0\in Y$.
 The specific nature of the generation of this sample $y_0$ is not relevant to what comes next. Then, it is natural to consider the constant IFS, where $\tau_\theta(y) = y_0$, for any $y\in Y$, $\theta \in \Theta$. We want to recover Bayes' rule of Section \ref{meas}. With this purpose, we explain also the connections between notations of both sections.
	
	\noindent
	In this case the prior items are:\newline
	1. a probability $d\theta$ on $\Theta$  \newline
	2. a density function $\pi_a:\Theta\to(0,+\infty)$ and so the probability $d\nu(\theta) := \pi_a(\theta)d\theta$\newline
	3. a loss function ${l}:\Theta\times Y\to(0,+\infty)$. We can suppose the remark \ref{mui} is satisfied and write also ${l}(\theta,y)=f(y|\theta)$, using the notation of Section \ref{meas}.
	
	By considering the intermediate item $\tau_\theta(y) = y_0$, the next intermediate items 5. and 6. are essentially unique: \newline
	5. The function $\psi = 1$ and so the canonical normalizer pair $\psi(y) =1$ and $\varphi(y)=\int l(\theta,y)\pi_a(\theta)d\theta$. So the $\nu$-Jacobian  is $\bar{{l}}(\theta,y) := \frac{{l}(\theta,y)}{\int {l}(\theta,y)\pi_a(\theta)d\theta}.$ \newline
	It follows that \[\bar{{l}}(\theta,y)= \frac{f(y|\theta)}{p(y)},\] where $p(y)$ is given in \eqref{normalizerp} in Section \ref{meas} and coincides with $\varphi(y)$.   \newline
	6. The stationary probability $\rho=\delta_{y_0}$  (which is the unique stationary probability for such IFS).
	
Now we can consider the posterior items. We observe that as for $\tau\equiv y_0$, the intermediate items $\psi,\varphi,\rho$ are essentially unique, we can consider that $\pi_p(\theta|\tau,\psi,\rho)$ corresponds in some sense to a computation of $\pi_p(\theta|y_0)$.

	\smallskip
	
The posterior items are given by:

\noindent	
	1. The probability
	\[d\pi = \bar{{l}}(\theta,y)\pi_a(\theta)d\theta d\rho(y) =
	 \frac{{l}(\theta,y)\pi_a(\theta)}{\int {l}(\theta,y)\pi_a(\theta)d\theta}\,d\theta\, d\delta_{y_0}(y) = \frac{f(y|\theta)\pi_a(\theta)}{p(y)}\,d\theta\, d\delta_{y_0}(y). \]
	
	 \noindent
2. The probability density kernel $\pi_p(.|\tau,\psi,y):\Theta\to\mathbb{R}$ given by
\[\pi_p(\theta|\tau,\psi,y) = \frac{l(\theta,y)\pi_a(\theta)}{\int l(\theta,y)\,\pi_a(\theta)d\theta}=\frac{f(y|\theta)\pi_a(\theta)}{p(y)}.\]	
	\noindent
3. The probability $\nu_p$, which is the $\theta-$marginal of $\pi$, satisfies
	\[d\nu_p(\theta) = \frac{{l}(\theta,y_0)\pi_a(\theta)}{\int {l}(\theta,y_0)\pi_a(\theta)d\theta}\,d\theta =\frac{f(y_0|\theta)\pi_a(\theta)}{p(y_0)}\,d\theta .\]
	
	\noindent
	4. The probability density function
	\[\pi_p(\theta|\tau,\psi,\delta_{y_0}) = \frac{l(\theta,y_0)\pi_a(\theta)}{\int l(\theta,y_0)\pi_a(\theta)d\theta}.\]
	Using the correspondence of notations we get
	\[\pi_p(\theta|\tau,\psi,\delta_{y_0}) = \frac{f(y_0|\theta)\pi_a(\theta)}{p(y_0)}.\]
This density function $\pi_p$ is exactly the posterior probability density function given the sample $y_0$, as presented in \eqref{posteriorBayes} when considered $f=l$. 
	
	We observe that $y_0$ is any fixed point of $Y$, and so we can also denote this generic point by $y$. Finally we can write that for a given sample $y$, the posterior density function is
\begin{equation}  \label{trew}
	\pi_p(\theta|y) = \frac{{l}(\theta,y)\pi_a(\theta)}{\int {l}(\theta,y)\pi_a(\theta)d\theta} = \frac{f(y|\theta)\pi_a(\theta)}{p(y)},
\end{equation}
where $\pi_a$ is the  prior probability density function. This corresponds to \eqref{posteriorBayes3}.
\end{example}

\begin{example}\label{maine2} In this example we consider the identity IFS, which is the one given by  $\tau_\theta(y) = y$, for any $y\in Y$, $\theta \in \Theta$. As in the above example, we want to recover Bayes' rule of Section \ref{meas}.
	
	\noindent
	In this case the prior items are:\newline
	1. a probability $d\theta$ on $\Theta$  \newline
	2. a density function $\pi_a:\Theta\to(0,+\infty)$ and so the probability $d\nu(\theta) := \pi_a(\theta)d\theta$\newline
	3. a loss function ${l}:\Theta\times Y\to(0,+\infty)$. We can suppose again that remark \ref{mui} is satisfied and write also ${l}(\theta,y)=f(y|\theta)$, using the notation of Section \ref{meas}. \newline
	The intermediate items are: \newline
	4. the IFS $\tau_\theta(y) = y$\newline
	5. The (essentially unique) normalizer pair $\psi=1$ and $\varphi(y) =\int {l}(\theta,y)\pi_a(\theta)d\theta$.  It follows, as in above example, that \[\bar{l}(\theta,y)= \frac{{l}(\theta,y)}{ \int {l}(\theta,y)\pi_a(\theta)d\theta}=\frac{{l}(\theta,y)}{ \int {l}(\theta,y)\pi_a(\theta)d\theta}=\frac{f(y|\theta)}{p(y)},\] where $p(y)$ is given in \eqref{normalizerp} in Section \ref{meas}.\newline
	6. Any probability $\rho$ on $Y$ (any one is stationary).
	
	\bigskip
	
	\noindent
	The posterior items are given below:\newline
	1. The probability
	\[d\pi =   \frac{{l}(\theta,y)\pi_a(\theta)}{\int {l}(\theta,y)\pi_a(\theta)d\theta}\,d\theta\, d\rho(y) = \frac{f(y|\theta)\pi_a(\theta)}{p(y)}\,d\theta\, d\rho(y).\]
	2. The probability density kernel $\pi_p(.|\tau,\psi,y):\Theta\to(0,+\infty)$ given by
	\[\pi_p(\theta|\tau,\psi,y) = \frac{l(\theta,y)\pi_a(\theta)}{\int l(\theta,y)\,\pi_a(\theta)d\theta}.\]	
	3. The probability $\nu_p$ which is the $\theta-$marginal of $\pi$
	\[d\nu_p(\theta) = \left[\frac{\int l(\theta,y)\pi_a(\theta)}{\int {l}(\theta,y)\pi_a(\theta)d\theta}d\rho(y)\right]\,d\theta =\left[\int \frac{f(y|\theta)\pi_a(\theta)}{p(y)}d\rho(y)\right]\,d\theta .\]
	4. The probability density function $\pi_p(\theta|\tau,\psi,\rho):\Theta\to\mathbb{R}$ which is given by
	\begin{equation} \label{poyu} \pi_p(\theta|\tau,\psi,\rho) = \int \frac{l(\theta,y)\pi_a(\theta)}{\int {l}(\theta,y)\pi_a(\theta)d\theta}d\rho(y)=\int \frac{f(y|\theta)\pi_a(\theta)}{p(y)}d\rho(y).
\end{equation}
	
  As we set $p(y)= \varphi(y)= \int l(\theta,y)\,\pi_a(\theta)d\theta $, we get in posterior item 2 above  the  {same} expression given by \eqref{posteriorBayes3}.
The  equation \eqref{poyu} does not match with \eqref{posteriorBayes} and \eqref{posteriorBayes3} because we have to consider a mean using $\rho$. While in \eqref{posteriorBayes} we are computing $\pi_p(\theta|y_0)$ for an observed sample $y_0$, in \eqref{poyu} we are computing  $\pi_p(\theta|\tau,\psi,\rho)$. We remark that for the identity IFS any probability $\rho$ is stationary and produces a different $\pi_p$. We can take, as a special example, $\rho=\delta_{y_{0}}$. { For this particular stationary prior item 6. we get that \eqref{poyu} match with \eqref{posteriorBayes} in Section \ref{meas}. Furthermore, such prior $\rho=\delta_{y_0}$ will be essential in Example \ref{zeze}.}

\end{example}

\begin{remark}\label{sample2}  Example \ref{meansample} describes an interpretation of the formula \eqref{poyu}.
\end{remark}

\begin{example} \label{marma} Markov Measures and the IFS  Bayesian method.
	
	\noindent
	Suppose that $\Theta = Y = \{1,...,d\}$ and let $\tau_\theta(y) = \theta$,  for all $y,\theta$. As we will see next it is natural to consider such  IFS. The introduction of the possibility to set a suitable $\rho $ over  $Y$, as considered in Intermediate item 6 in Definition \ref{maindefinition}, will be necessary to get natural posterior items. 
	
	Consider the prior items:\newline
	1. let $d\theta $ be the counting measure (not a probability)\newline 	
	2. $\pi_a\equiv 1$ and so  $d\nu(\theta) = d\theta$\newline
	3. a positive loss function ${l}:\Theta\times Y\to\mathbb{R}$\newline
	Consider the intermediate items below: \newline
	4. the IFS $\tau_\theta(y) = \theta$\newline
	5. There exists a lot of normalizer pairs, but we will consider a special one by taking $\varphi$  constant. From the Perron-Frobenius theorem, there exists a (unique) positive number $\lambda$ and a (unique except by a multiplicative constant) positive vector $h$ such that
	\[\sum_{\theta}h(\theta){l}(\theta,y) = \lambda\cdot h(y)\,\forall y.\]
	We set the normalizer pair $\psi=h$ and $\varphi(y)\equiv \lambda$. Therefore,
\begin{equation} \bar{l}(\theta,y) = \frac{l(\theta,y)\,\psi(\theta)}{\lambda \,\psi(y)\,}=\frac{l(\theta,y)\psi(\theta)}{\sum_\theta l(\theta,y)\psi(\theta)}.
	\end{equation}
	We observe that $\bar{{l}}$ to  be a Jacobian means that it defines  a column stochastic matrix:
	\[\sum_\theta \bar{l}(\theta,y) = 1,\,\forall\,y.\]
	6. A stationary probability $\rho=(\rho_1,...,\rho_d)$ on $Y$ satisfies, for any $g$
	\[\sum_{j}\sum_\theta \bar{l}(\theta,j)g(\theta)\rho_j = \sum_j (g(j))\rho_j\]
	Choosing for each $i\in Y$ the function $v=I_i$ (it is a basis for the space of functions on $Y$) we get that $\rho$ needs to satisfy
	\[\sum_j \bar{l}(i,j)\rho_j = \rho_i, \,\forall\,i,\] that is, the probability vector $\rho$ satisfies
	\[ (\bar{l})\cdot \rho =\rho.\]
	Such a $\rho$ is unique (and it is usually known in the literature of Stochastic Processes as the stationary probability vector associated with the column stochastic matrix $(\bar{l})$).
	
	\bigskip
	
	\noindent
	The posterior items are:\newline
	1. The Markov measure for cylinders of length two:
	\[\pi(i,j) = \bar{l}(i,j)\cdot \rho(j)\]In this way, for any function $g$,
	\[\sum_{i,j}g(i,j)\pi(i,j) = \sum_{i,j} \bar{l}(i,j)g(i,j)\rho(j).\]
	2. The probability density kernel $\pi_p(.|\tau,\psi,y):\Theta\to\mathbb{R}$ given by
	\[\pi_p(\theta|\tau,\psi,y) = \frac{l(\theta,y)h(\theta)}{\sum_\theta l(\theta,y)\,h(\theta)}=\bar{l}(\theta,y).\]	
	3. The probability $\nu_p$, which is the $\theta-$marginal of $\pi$,
	$$\nu_p(i) = \sum_j \bar{l}(i,j)\cdot \rho(j) = \rho(i).$$
	4. The probability density function $$\pi_p(i) = \sum_j \bar{l}(i,j)\cdot \rho(j)=\rho(i).$$

\end{example}

\begin{remark} The posterior probability density kernel $\pi_p(\theta|\tau,\psi,y) = \bar{l}(\theta,y)$ coincides with the transition probability $P(\theta|y)$ which is given by the column stochastic matrix $\bar{l}$  (a  natural way  to define a   Markovian stochastic process taking values on a finite set). The above  $\pi_p$ represents  Bayes'  rule according to \eqref{Artur} and \eqref{IFSBayes}.
\end{remark}

\section{Examples concerning contractible IFS } \label{dynIFS}

In this section, we present some more examples concerning the generalized Bayes' rule and the IFS Bayesian method in a dynamic context. With this purpose we assume in such examples that $\Theta,Y$ are compact metric spaces, respectively, with metrics $d_1$ and $d_2$, and we consider on $\Theta\times Y$ the metric
$$d((\theta_1,y_1),(\theta_2,y_2)) = d_1(\theta_1,\theta_2)+d_2(y_1,y_2).$$

 We will present a dynamical version of  Bayes' rule for  an IFS $\tau$, in consonance with the reasoning of \cite{MO},  and  assuming some regularity for the loss function. The meaning of this statement is associated with the fact that in this section  we will assume conditions such that the normalizer pair $(\varphi,\psi)$ for $({l},\nu,\tau)$,  can be taken in  such way that $\varphi$ is constant.
\smallskip

The next example will  help  the reader to make connections with classical Thermodynamic Formalism as described in \cite{PP}.

\begin{example} \label{trite} General equilibrium measures in the full shift
	
	\noindent
	Suppose that $\Theta=\{1,...,d\}$, $Y = \{1,...,d\}^{\mathbb{N}}$ and let $\tau_\theta(y) = (\theta,y_1,y_2,...)$. In such a setting the space $\Theta \times Y$ is also denoted by $\Omega$ and we have an invertible map $\Gamma:\Omega \to Y$ given by $\Gamma(\theta,(y_1,y_2,...)) = (\theta,y_1,y_2,y_3,...)$. Such a map allows us to identify $\Omega$ as $Y$ and consequently any function or probability on $\Omega$ can be seen as a function or probability on $Y$, respectively.   The loss function $l(\theta,y)$ we will consider here is of the form $l(\theta, y)=e^{v(\tau_\theta (y))},$ for some Lipschitz  function $v:Y \to \mathbb{R}.$
	
	Consider the prior items below:\newline
	1. let $d\theta $ be the counting measure (not a probability) on $\Theta$\newline 	
	2. $\pi_a\equiv 1$ and so  $d\nu(\theta) = d\theta$\newline
	3. a Lipschitz function $v:Y \to\mathbb{R}$ and the loss function $l:\Omega \to\mathbb{R}$, given by $l(\theta,y) = e^{v(\tau_{\theta}(y))}.$\newline
 Consider the following intermediate items:\newline
	4. The IFS $\tau_\theta(y) = (\theta,y_1,y_2,...)$ as presented above.\newline
	5. From the Ruelle-Perron-Frobenius theorem (see \cite{PP}), there exists a unique positive number $\lambda$ and a unique (except by a multiplicative constant)   positive Lipschitz  function $h: Y \to \mathbb{R}$  such that
\begin{equation} \label{qued}
	\sum_{\theta}l(\theta,y)h(\tau_\theta( y))=\sum_{\theta}e^{v(\theta, y_1, y_2,... ) }h(\theta, y_1, y_2,...) = \lambda\cdot h(y)\,\forall y.\end{equation}
As  $h(\theta y)=h(\tau_\theta(y))$ we get a normalizer pair by setting $\psi=h$ and $\varphi(y)\equiv \lambda$.
In this way
\begin{equation} \label{qued}\overline{l}(\theta,y)= \frac{l(\theta,y)h(\theta y)}{\sum_\theta l(\theta,y)h(\theta y)}= \frac{l(\theta,y)h(\theta y)}{\lambda h(y)}.
\end{equation}
is a Jacobian. \newline 	
	6. There exists a unique stationary probability $\rho$ on $Y$ satisfying, for any $g$
	\[\int_Y\sum_\theta \bar{l}(\theta,y)g(\theta y)d\rho(y) = \int g(y)d\rho(y).\]  The probability $\rho$ on $Y$ is invariant for the shift map $\sigma$, and it is the equilibrium probability  for the function $v:Y \to \mathbb{R}$ (in the sense of \cite{PP}).

	We can also identify $\rho$ as a probability on $\Omega$ using the map $\Gamma$.
	
	\bigskip
	
	\noindent
	The posterior items are:\newline
	1. The equilibrium probability $\pi=\rho$ (considering the identification of $\Omega=\Theta\times Y$ and $Y$ given by $\Gamma$):
	\[\int g(\theta, y)d\pi(\theta,y) = \int_Y \sum_\theta \bar{l}(\theta y)g(\theta y)d\rho(y)= \int g(y)d\rho(y)=\int g\,d\rho.\]
 	2. The probability density kernel $\pi_p(.|\tau,\psi,y):\Theta\to\mathbb{R}$ given by
	\[\pi_p(\theta|\tau,\psi,y) = \frac{l(\theta,y)h(\theta y)}{\sum_\theta l(\theta,y)\,h(\theta y)}=\bar{l}(\theta,y).\]	
	3. The probability $\nu_p$, which is the $\theta-$marginal of $\pi$,
	$$\nu_p(i) = \int_Y \bar{l}(i,y)d\rho(y) = \pi([i]).$$
4. The probability density function $$\pi_p(i) = \pi([i]).$$
	Above we are denoting by $[i]$ a cylinder set of length one.
	
The above function $\bar{l}$ is sometimes called the Jacobian of the equilibrium probability $\rho$. The introduction of the possibility to set a convenient probability  $\rho $ over  $Y$, which was considered in Intermediate item 6 in Definition \ref{maindefinition}, was an important issue to be able to get the posterior probability  $\pi_p$ in posterior item 2. above. The important issue in this example is to get the posterior probability. The law $\pi_p$ represents Bayes'  rule according to \eqref{Artur} and \eqref{IFSBayes} in this case.
\end{example}

	\noindent
	\begin{definition} \label{contr}
	Suppose that $\Theta$ and $Y$ are compact metric spaces.  We say that the IFS $\tau$ is contractible (or satisfies the uniform contraction property) when there exists $0<\gamma<1$, such that,
	\[d(\tau_{\theta_1}(y_1) , 	\tau_{\theta_2}(y_2))\leq \gamma [d_1(\theta_1,\theta_2)+d_2(y_1,y_2)]\]
	for any $\theta_i\in\Theta$ and $y_i\in Y$.
	\end{definition}
	
	\bigskip
	
\begin{example}\label{exholonomic} Holonomic probabilities for contractible IFS (see \cite{MO}).

	\noindent
	Consider the following prior items:\newline
	1. a probability $d\theta$ on $\Theta$  \newline
	2. a probability density function $\pi_a:\Theta\to(0,+\infty)$ and so $d\nu(\theta) := \pi_a(\theta)d\theta$\newline
	3. a Lipschitz continuous function $\mathfrak{l}:\Theta\times Y\to\mathbb{R}$ and the loss function  $l:\Theta\times Y \to(0,+\infty)$, given by $l(\theta,y) = e^{\mathfrak{l}(\theta,y)}.$\newline
 Consider the following intermediate items:\newline
	4. a contractible IFS $\tau$.\newline
	5. we will consider a special normalizer pair where $\varphi$ is constant. Following \cite{MO} there exists a unique positive number $\lambda$ and a unique (except by a multiplicative constant) positive function $h:Y\to\mathbb{R}$ such that
	\[\int e^{\mathfrak{l}(\theta,y)}h(\tau_\theta( y)) d\nu(\theta)= \lambda\cdot h(y)\,\forall y.\]
	As  in the above example, we get a normalizer pair by setting $\psi=h$ and $\varphi(y)\equiv \lambda$.
	In this way
	\begin{equation} \label{qued}\overline{l}(\theta,y)= \frac{l(\theta,y)h(\tau_\theta( y))}{\sum_\theta l(\theta,y)h(\tau_\theta(y))}= \frac{l(\theta,y)h(\tau_\theta (y))}{\lambda h(y)}
	\end{equation}
	is a Jacobian.
	 \newline
	6. Following \cite{MO} there exists a unique stationary probability $\rho$ for $(\bar{l},\nu,\tau)$. Furthermore, if we denote by $L:C(Y)\to C(Y)$ the operator given by
	\[L(g)(y) = \int \bar{l}(\theta,y)g(\tau_\theta(y))d\nu(\theta),\] then
	$L^{n}(g)$ converges uniformly to $\int g d\rho$ as $n\to+\infty$.
		
	The proofs of all claims above appear in \cite{MO}. The introduction of  $\rho $, which was considered in Intermediate item 6 in Definition \ref{maindefinition}, is an important issue to get the posterior items.
	
	\bigskip
	
	\noindent
	The posterior items are:\newline
	1. The probability $\pi = \bar{l}(\theta,y)\pi_a(\theta)d\theta d\rho(y)$. We observe that $\pi$ satisfies
	\[\int g(\theta,y)\,d\pi = \int_Y \frac{{l}(\theta,y)h(\tau_\theta(y))}{\lambda h(y)}\pi_a(\theta) g(\theta,y)d\theta d\rho(y).\]
2. The probability density kernel $\pi_p(.|\tau,\psi,y):\Theta\to\mathbb{R}$ given by
\[\pi_p(\theta|\tau,\psi,y) = \frac{l(\theta,y)h(\tau_\theta( y))}{\lambda\,h(y)}\pi_a(\theta)=\bar{l}(\theta,y)\pi_a(\theta).\]	
	3. The probability $\nu_p$ which is the $\theta-$marginal of $\pi$.
	On this way
	\[\int g(\theta)d\nu_p(\theta) = \int_\Theta \left[\int_Y \frac{l(\theta,y)h(\tau_\theta( y))\pi_a(\theta)}{\lambda\,h(y)}\pi_a(\theta) d\rho(y)\right]g(\theta) d\theta\]
	3. The density function $\pi_p:\Theta\to\mathbb{R}$ which satisfies $d\nu_p(\theta) = \pi_p(\theta)d\theta$. Explicitly it is given by
	\[\pi_p(\theta)=\int_Y \frac{l(\theta,y)h(\tau_\theta( y))}{\lambda\,h(y)}\pi_a(\theta) d\rho(y).\]

\end{example}

\bigskip

\section{Entropy, pressure and the principle of minimization of information} \label{epi}

In this section, we consider the prior, intermediate and posterior items as in Definition \ref{maindefinition}. We will show that the posterior items $\pi_p$ and $\pi$ are optimal with respect to a certain variational principle.  In Thermodynamic Formalism (see \cite{PP}), and also in the holonomic setting (see \cite{MO},\cite{LM2}), it corresponds to finding the equilibrium probability maximizing the pressure functional.
In information theory, this formalism corresponds to an information conservation principle and an optimal information processing rule \cite{Ze}.

The main goal of \cite{Ze} was  to derive Bayes' rule as a consequence of the minimization of the information procedure. The reasoning of \cite{Ze} was our main motivation for the present paper.
Minimization of information can be seen as a variational principle for pressure (as in \cite{PP}, \cite{MO}).

We start by recalling  the definition of entropy following the discussion present in \cite{LM2}. 
The entropy of a probability $\pi$ on $\Theta\times Y$ with respect to a probability $\nu$ on $\Theta$ is given by
\[H^{\nu}({\pi}) = -\sup\{\int \log(J)\,d\pi\,|\, J\, \text{is a}\, \nu- \text{Jacobian}\}.\] This supremum is taken over functions $J$ such that $\log(J)$ has a well defined integral with respect to $\pi$. As $J=1$ is a $\nu-$Jacobian we  get $H^{\nu}(\pi)\leq 0$.
If $\rho$ is the $y$-marginal of $\pi$ then
\[H^\nu(\pi) = -D_{KL}(\pi|\nu\times \rho),\]
where $D_{KL}$ is the Kullback-Leibler divergence.    It means that
\begin{equation}\label{KL} H^{\nu}(\pi) =\left\{\begin{array}{ll} -\int \log(J)\,d\pi & \text{if}\,\, d\pi = J(\theta,y)d\nu(\theta) d\rho(y)\\ \\
	-\infty& \begin{array}{c}\text{if}\,\pi\,\text{is not absolutely continuous}\\\text{ with respect to}\,\nu\times\rho\end{array} \end{array}\right.\end{equation}
In such an entropy, the probability $\pi$ and its $y$-marginal $\rho$ are any probabilities and they do not need to be holonomic or stationary.

Let us suppose that there exists a constant $a>0$ such that $a<l(\theta,y)$ for all $(\theta,y)\in\Theta\times Y$ and $a<\psi(y)$ for all $y\in Y$. As these functions are also bounded (see Definition \ref{maindefinition}) it follows from \eqref{barl} that $\bar{l}$ satisfies a similar property and it is bounded. If $\Theta$ and $Y$ are compact metric spaces and the functions $l,\psi$ are continuous then these hypotheses are satisfied. Furthermore we suppose $\pi_a$ is bounded. If $\rho$ is stationary for $(\bar{l},\nu,\tau)$ and $\pi$ is the posterior probability, then (see \cite{LM2})
\[ \int \log(\bar{l}) \,d{\pi} +H^{\nu}({\pi}) = 0 = \sup_{\tilde{\pi}} \int \log(\bar{l} )\,d\tilde{\pi} +H^{\nu}(\tilde{\pi}).\]
If $\bar{l}(\theta,y)$ is a $\nu$-Jacobian and $d\nu = \pi_a(\theta)d\theta$ then $\bar{l} \cdot \pi_a$ is a $d\theta$-Jacobian and as $d\pi =(\bar{l} \cdot \pi_a) d\theta d\rho$  we also get
\[ \int [\log(\bar{l}) +\log(\pi_a)] \,d{\pi} +H^{d\theta}({\pi}) = 0 = \sup_{\tilde{\pi}} \int [\log(\bar{l})+\log(\pi_a)] \,d\tilde{\pi} +H^{d\theta}(\tilde{\pi}).\]
As $\bar{l}(\theta,y)=\frac{{l}(\theta,y)\psi(\tau_\theta(y))}{\psi(y)\varphi(y)}$ and the integral of $\log(\psi(\tau_\theta(y))) - \log(\psi(y))$ over holonomic probabilities is equal to zero, we also get
\begin{align*} 0&= \int [\log({l}(\theta,y)) -\log(\varphi(y))] \,d{\pi} +H^{\nu}({\pi})\\
	&= \sup_{\tilde{\pi}\,holonomic} \int [\log({l}(\theta,y))-\log(\varphi(y))] \,d\tilde{\pi} +H^{\nu}(\tilde{\pi}).
	\end{align*}

Furthermore, we obtain the following main result.

\begin{theorem}\label{teo_eq_pressure} Consider the prior and intermediate items of the IFS Bayesian method, according to Definition \ref{maindefinition}. Furthermore, suppose that $\pi_a$ is bounded and that there exists a constant $a>0$ such that $a<l(\theta,y)$ for all $(\theta,y)\in\Theta\times Y$ and $a<\psi(y)$ for all $y\in Y$.  Let us to consider the following variational problem
\begin{equation} \label{eq_pressure} \sup_{\tilde{\pi}\,holonomic} \int [\log({l}(\theta,y))+\log(\pi_a(\theta))-\log(\varphi(y))] \,d\tilde{\pi} +H^{d\theta}(\tilde{\pi}).
\end{equation}
This supremum is attaining at any\footnote{If there are more than one stationary $\rho$   with respect to $(\bar{{l}},\nu,\tau)$, then we get more than one possible posterior probability $\pi$} posterior probability $\pi$. Furthermore, it is equal to zero.
\end{theorem}
\smallskip

The next example shows how the reasoning of \cite{Ze} fits our IFS Bayesian method.
\smallskip

\begin{example}  \label{zeze} Optimal information processing rule as in \cite{Ze}.

\noindent  Assume the prior items for the IFS Bayesian method as given in Example \ref{maine} or \ref{maine2}. Furthermore, assume that the assumption in Remark \ref{mui} is satisfied.  We observe that the function
${l}(\theta,y) $   is corresponding here to the $ f(y|\theta)$ presented in section \ref{meas} and in \cite{Ze}.
For a such class of examples, we can consider the canonical normalizer pair, that is $\psi=1$ and 	
	\[ \varphi(y) = \int {l}(\theta,y)\pi_a(\theta)d\theta=
\int f(y|\theta) \pi_a(\theta)d\theta. \]
We point out that the probability density function $p$ in \cite{Ze} coincides with our $p$, as defined in \eqref{ljk3}, and moreover $p(y)=\varphi(y)$.

Consider also the stationary probability $\rho=\delta_{y_0}$ and the associated posterior probability $\pi$ satisfying
\[d\pi = \pi_p(\theta|y_0)d\theta d\delta_{y_0}=\frac{f(y_0|\theta)\pi_a(\theta)}{p(y_0)}\,d\theta\, d\delta_{y_0}(y).\]
Such $\pi$ attains the supremum in \eqref{eq_pressure} over all holonomic probabilities $\tilde{\pi}$.
In particular, as the $y$-marginal of $\pi$ is $\delta_{y_0}$, it also assumes such supremum over the subset of holonomic probabilities with $y-$marginal equal to $\delta_{y_0}$. Any such holonomic probability has the form $d\tilde{\pi}=\tilde{\pi}_p(\theta)d\theta d\delta_{y_0}$, where $\tilde{\pi}_{p}(\theta)$ is a probability density function, or else $H^{d\theta}(\tilde{\pi})=-\infty$ and certainly it does not attains the supremum. Restricted to such particular class of  probabilities, that is, holonomic probabilities $\tilde{\pi}$ of the form $d\tilde{\pi}=\tilde{\pi}_p(\theta)d\theta d\delta_{y_0}$, the variational problem \eqref{eq_pressure} of Theorem \ref{teo_eq_pressure} can be rewritten as\footnote{observe the importance of considering $d\tilde{\pi}=\tilde{\pi}_p(\theta)d\theta d\delta_{y_0}$, to get the below expression  from \eqref{eq_pressure}, and that the supremum is over $\tilde{\pi}_p(\theta)$, which is a probability density function and no more a probability; furthermore, for the last term we take into account   \eqref{KL}.  }
	\[ \sup_{\tilde{\pi}_p(\theta)} \int [\log({l}(\theta,y_0))   +\log(\pi_a(\theta)) ]\tilde{\pi}_p(\theta)d\theta - \log(p (y_0))  - \int \log(\tilde{\pi}_p(\theta))\tilde{\pi}_p(\theta)d\theta.\]	
		The above variational problem is exactly equal (changing sign and just replacing the notation of sample from letter $y_0$ by the letter $y$) to the minimization problem (2.6) in \cite{Ze}. Furthermore, from Theorem \ref{teo_eq_pressure} and Example \ref{maine} its supremum value is attained at the posterior probability $\pi=\pi_p(\theta|y_0)d\theta d\delta_{y_0}$, where $\pi_p(\theta|y_0)=\frac{f(y_0|\theta)\pi_a(\theta)}{p(y_0)}$ is given by Bayes' rule \eqref{posteriorBayes}, and its value is equal to zero, that is, 	
	\[0= \int [\log({l}(\theta,y_0))   +\log(\pi_a(\theta)) ]\pi_p(\theta|y_0)d\theta - \log p (y_0)  - \int \log(\pi_p(\theta|y_0))\pi_p(\theta|y_0)d\theta. \]
This corresponds exactly to the solution in the form of Bayes' rule (2.8) in \cite{Ze}, which solves the variational problem (2.6), and also  the one in \cite{Ze}, which is called an optimal information processing rule.

\end{example}

\begin{example}\label{pressureholonomic} Holonomic probabilities for contractible IFS
	
Consider the notations of Example \ref{exholonomic}. In this case $l=e^{\mathfrak{l}}$. Choosing  $\varphi$ as a positive constant (which is the eigenvalue $\lambda$)  we get from Theorem \ref{teo_eq_pressure}
	\[ 	 \int \mathfrak{l}(\theta,y) \,d{\pi} +H^{d\theta}({\pi}) = \log(\lambda) = \sup_{\tilde{\pi}\,holonomic} \int {\mathfrak{l}(\theta,y)} \,d\tilde{\pi} +H^{d\theta}(\tilde{\pi}),\]
	which is the   variational principle for pressure as considered in \cite{MO}. The equivalence between the entropy as defined in \cite{MO} and the above-defined entropy is proved in \cite{LM2}.
	
\end{example} 	

\bigskip

\noindent 
Artur O. Lopes\newline 
Instituto de Matemática e Estatística - UFRGS - Brazil\newline 
arturoscar.lopes@gmail.com

\bigskip

\noindent
Jairo K. Mengue\newline 
Departamento Interdisciplinar - UFRGS - Brazil\newline
jairo.mengue@ufrgs.br


\begin{thebibliography}{}


\bibitem{Abra}
F. Abramovich and Y. Ritov, Statistical Theory, A Concise Introduction, CRC Press (2013)



\bibitem{DLL} M. Denker, A. O. Lopes  and S. R. C. Lopes,
Dynamical hypothesis tests and Decision Theory for Gibbs distributions, Discrete and Continuous Dynamical Systems - Series A - Vol. 43, No. 5, pp. 1942-1958 (2023)







\bibitem{FLL} H. H. Ferreira, A. O. Lopes and S. R. C. Lopes,
Decision Theory and Large Deviations for  Dynamical Hypotheses Tests: the Neyman-Pearson Lemma, min max and Bayesian tests, \emph{Journal of Dynamics and Games} Volume 9, Number 2, April 2022 - pp 125--150

\bibitem{HW}
C. Holmes and S. Walker,
A general framework for updating belief distributions,
\emph{J. R. Stat. Soc. Ser. B. Stat. Methodol}. 7, no. 5, 1103–-1130 (2016)

\bibitem{LLV} A. O. Lopes, S. R. C. Lopes and P. Varandas.
Bayes posterior convergence for loss functions via almost additive Thermodynamic Formalism, Journ. of Statis. Physics,  186:35 (2022)



\bibitem{LM2} A. O. Lopes  and J. K. Mengue,
On information gain, Kullback-–Leibler divergence, entropy production and the involution kernel, Disc. and Cont. Dyn. Syst. Series A, Vol. 42, No. 7, 3593–-3627 (2022)


\bibitem{LO} A. O. Lopes and E.  Oliveira,
Entropy and variational principles for holonomic probabilities of IFS, Disc. and Cont. Dyn. Systems (Series A), Vol 23, N. 3, 937--955 (2009)

\bibitem{MO} J. Mengue and E. Oliveira, Duality results for iterated function systems with a general family of branches, \emph{Stochastics and Dynamics} Vol. 17, No. 03, 1750021 (2017).
	





\bibitem{Nobel1} K. McGoff, S. Mukherjee and A. Nobel, Gibbs posterior convergence and Thermodynamic formalism, Ann. Appl. Probab. 32, no. 1, 461--496  (2022)
	
\bibitem{PP}
W. Parry and M.  Pollicott, Zeta functions and the periodic
orbit structure of hyperbolic dynamics, \emph{Ast\'erisque}
Vol {187-188}, 1990.
\bibitem{Ze}
A. Zellner,  Optimal information processing and Bayes's theorem. \emph{Amer. Statist.} 42 (1988), no. 4, 278–-284.



\end{thebibliography}
\end{document}